\documentclass{amsart}
\usepackage{amsmath}
\usepackage{amssymb}
\usepackage{amsthm}
\usepackage{dsfont}
\usepackage{lineno}
\usepackage{subfigure}
\usepackage{graphicx}
\usepackage{multirow}
\usepackage{color}
\usepackage{xspace}
\usepackage{pdfsync}
\usepackage{hyperref} 
\usepackage{algorithmicx}
\usepackage{algpseudocode}
\usepackage{algorithm}
\usepackage{mathtools}
\usepackage{enumitem}
\graphicspath{{Figures/}}

\DeclareMathOperator*{\argmin}{argmin}

\newcommand{\bq}{\begin{equation}}
\newcommand{\eq}{\end{equation}}
\newcommand{\R}{\mathbb{R}}

\newcommand{\abs}[1]{\left\vert#1\right\vert}
\newcommand{\norm}[1]{\left\Vert#1\right\Vert}

\newcommand{\G}{\mathcal{G}}
\newcommand{\bO}{\mathcal{O}}

\newcommand{\Dt}{\mathcal{D}}

\newcommand{\Sf}{\mathbb{S}^{2}}
\newcommand{\Tf}{\mathcal{T}}
\newcommand{\Zf}{\mathcal{Z}}

\newcommand{\Ef}{\mathcal{E}}
\newcommand{\cp}{\text{cp}}
\newcommand{\MA}{Monge-Amp\`ere\xspace}

\algnewcommand{\LineComment}[1]{\State \(\triangleright\) #1}

\newtheorem{theorem}{Theorem}

\theoremstyle{lemma}
\newtheorem{lemma}[theorem]{Lemma}
\newtheorem{lem}[theorem]{Lemma}

\newtheorem{definition}[theorem]{Definition}

\newtheorem{remark}[theorem]{Remark}

\newtheorem{hypothesis}[theorem]{Hypothesis}

\theoremstyle{remark}

\makeatletter
\newcommand\appendix@section[1]{%
\refstepcounter{section}%
\orig@section*{Appendix \@Alph\c@section: #1}%
}
\let\orig@section\section
\g@addto@macro\appendix{\let\section\appendix@section}
\makeatother

\begin{document}

\title[Optimal transport on the sphere]{A convergence framework for optimal transport on the sphere}

\author{Brittany Froese Hamfeldt}
\address{Department of Mathematical Sciences, New Jersey Institute of Technology, University Heights, Newark, NJ 07102}
\email{bdfroese@njit.edu}
\author{Axel G. R. Turnquist}
\address{Department of Mathematical Sciences, New Jersey Institute of Technology, University Heights, Newark, NJ 07102}
\email{agt6@njit.edu}

\thanks{The first author was partially supported by NSF DMS-1619807 and NSF DMS-1751996. The second author was partially supported by  an NSF GRFP}

\begin{abstract}
We consider a PDE approach to numerically solving the optimal transportation problem on the sphere. We focus on both the traditional squared geodesic cost and a logarithmic cost, which arises in the reflector antenna design problem.  At each point on the sphere, we replace the surface PDE with a generalized Monge-Amp\`ere type equation posed on the tangent plane using normal coordinates.  The resulting nonlinear PDE can then be approximated by any consistent, monotone scheme for generalized Monge-Amp\`ere type equations on the plane.  Existing techniques for proving convergence do not immediately apply because the PDE lacks both a comparison principle and a unique solution, which makes it difficult to produce a stable, well-posed scheme. By augmenting the discretization with an additional term that constrains the solution gradient, we obtain a strong form of stability. A modification of the Barles-Souganidis convergence framework then establishes convergence to the mean-zero solution of the original PDE.\end{abstract}

\date{\today}    
\maketitle
We consider the problem of optimal transportation on the sphere.  That is, given two prescribed density functions $f_1$ and $f_2$, we seek a mapping $T:\Sf\to\Sf$ such that
\bq\label{eq:OT}
T = \argmin\limits_{T_\# f_1 = f_2} \int_{\Sf} c(x,T(x)) f_1(x) dS(x).
\eq
Here $c(x,y)$ is the cost of transporting a unit of mass from $x$ to $y$ and $T_\# f_1 = f_2$ indicates that
\[ \int_A f_1(x)\,dS(x) = \int_{T(A)} f_2(y)\,dS(y) \]
for every measurable $A \subset \Sf$.

Perhaps the simplest cost is the squared geodesic distance
\[ c(x,y) = \frac{1}{2}d_{\mathbb{S}^2}(x,y)^2, \]
where $d_{\Sf}(x,y)$ denotes the geodesic distance between $x,y\in\Sf$.  This cost function has recently been applied to the problem of mesh generation on the sphere in the context of meteorology~\cite{McRae_OTonSphere,Weller_OTonSphere}.

A second cost of particular interest is the log cost
\[ c(x,y) = -\log \norm{x-y}, \]
which arises in the reflector antenna design problem~\cite{GlimmOliker_SingleReflector,Wang_Reflector2}. The notation $\norm{\cdot}$ denotes the Euclidean distance in the ambient space $\mathbb{R}^3$.

In the past several years, several new methods have been introduced to solve the optimal transportation problem in Euclidean space.  Most of these have been restricted to the quadratic cost function~\cite{BenamouDuval,BFO_OTNum,FroeseTransport,Levy_OT,Rubinstein_OT,Prins_OT}.  A few methods are available for problems with non-quadratic cost including linear programing methods~\cite{Schmitzer_OT} and a least-squares method introduced for a non-quadratic cost problem in geometric optics~\cite{Yadav_MA}.

Recently, some progress has been made in the solution of the optimal transport problem on the sphere.  The work of~\cite{Weller_OTonSphere} used a geometric interpretation of a \MA type equation on the sphere to produce the first such method, which applies to the squared geodesic cost.  A finite element solution of this \MA type equation was produced in~\cite{McRae_OTonSphere}. For problems posed on a subset of the sphere, the stereographic projection can be used to reframe the problem as an optimal transport problem on the plane (with non-quadratic cost); this was the approach of~\cite{RomijnSphere}. For a particular logarithmic cost function, the semi-discrete optimal transportation problem on the sphere admits a particularly nice interpretation in terms of generalized (spherical) power diagrams.  The work of~\cite{Cui_sphericalOT} recently exploited this interpretation to develop a fast, convergent method using techniques from computational geometry.

While several numerical methods have been proposed, and proof of convergence is sometimes possible in special cases, we are not aware of any general techniques for proving the convergence of PDE based methods for optimal transportation on the sphere.  The problem possesses several challenges that prevent the direct use of existing techniques.  (1) The curved geometry requires careful interpretation of the terms in the PDE operator. (2) Because solutions of the PDE are unique only up to additive constants, naive discretizations typically lead to schemes that are ill-posed and may not have any solution~\cite{HL_LagrangianGraphs}. This structure also makes it very challenging to establish the stability of approximation schemes.  (3) The PDE has no comparison principle, which precludes the direct use of the Barles-Souganidis convergence framework.  (4) The domain has no boundary, and thus boundary conditions cannot be used to build in the required stability and well-posedness as has been previously done for optimal transport problems in Euclidean space~\cite{HamfeldtBVP2}. 

In this article, we produce a new convergence framework for numerical methods for optimal transportation on the sphere.  Moreover, the approach is flexible,  encompassing both the squared geodesic cost and the log cost, with the potential to easily extend to other cost functions.  The method involves discretizing a \MA type equation on the sphere.  At each point on the sphere, we relate this to an equivalent PDE on the tangent plane through a careful choice of local coordinates that preserve the structure of the PDE operator. The resulting equation can be discretized using monotone generalized finite difference approximations and can be utilized for a wide variety of grids. 
The scheme is augmented with a constraint on the solution gradient and a careful shift of the resulting discrete solution.  These modifications yield a strong form of stability that allows us to modify the Barles-Souganidis framework to prove convergence.  

\section{Background}\label{sec:background}
\subsection{Optimal transport on the sphere}

We consider points $x,y$ lying on a unit sphere $\Sf$ centered at the origin.  We are interested in two different cost functions $c(x,y)$: the squared geodesic distance on the sphere,
\bq\label{eq:squareCost}
c(x,y) =  \frac{1}{2}d_{\Sf}(x,y)^2 = \frac{1}{2}\left(2\sin^{-1}\left(\frac{\norm{x-y}}{2}\right)\right)^2,
\eq
and the log-cost arising in the reflector antenna problem,
\bq\label{eq:logCost}
c(x,y) = -\log \norm{ x-y }.
\eq

The optimal map corresponding to each cost function is determined from the conditions
\vspace*{-4pt}
\bq\label{eq:mapConditionSphere}
\begin{cases}
\nabla_{\mathbb{S}^2,x}c \left( x,T(x,p) \right) = -p, & x\in\Sf, p\in \Tf_x\\
T(x,p) \in \Sf
\end{cases} \vspace*{-4pt}
\eq
where $\Tf_x$ denotes the tangent plane at $x$. The solution to the optimal transport problem is then given by
\bq\label{eq:PDE1}
F \left( x,\nabla_{\mathbb{S}^2} u(x), D_{\mathbb{S}^2}^2u(x) \right) = 0
\eq
where
\bq\label{eq:OTPDE}
F(x,p,M) \equiv -\det \left( M+A(x,p) \right) + H(x,p)
\eq
subject to the $c$-convexity condition, which requires
\bq\label{eq:cconvex}
D_{\mathbb{S}^2}^2u(x) + A(x,\nabla_{\mathbb{S}^2} u(x)) \geq 0.
\eq
Here
\begin{align*}
A(x,p) &= D_{\mathbb{S}^2,xx}^2c \left( x,T(x,p) \right)\\
H(x,p) &= \abs{\det{D_{\mathbb{S}^2,xy}^2c \left( x,T(x,p) \right)}}f_1(x)/f_2 \left( T(x,p) \right),
\end{align*}
and the PDE now describes a nonlinear relationship between the surface gradient and Hessian on the sphere.


\subsection{Regularity}
We consider the optimal transport problem~\eqref{eq:OT} under the following two sets of hypotheses.
\begin{hypothesis}[Conditions on data (smooth)]\label{hyp:Smooth}
We require problem data to satisfy the following conditions:
\begin{enumerate}
\item[(a)] There exists some $m>0$ such that $f_2(x) \geq m$ for all $x\in\Sf$.
\item[(b)] The mass balance condition holds, $\int_{\Sf} f_1(x)\,dx = \int_{\Sf} f_2(y)\,dy$.
\item[(c)] The cost function is either $c(x,y) = \frac{1}{2}d_{\Sf}(x,y)^2$ or $c(x,y) = -\log \norm{ x-y }$.
\item[(d)] The data satisfies the regularity requirements $f_1, f_2 \in C^{1,1}(\Sf)$.
\end{enumerate}
\end{hypothesis}

\begin{hypothesis}[Conditions on data (non-smooth)]\label{hyp:Nonsmooth}
We require problem data to satisfy the following conditions:
\begin{enumerate}
\item[(a)] There exists some $m>0$ such that $f_2(x) \geq m$ for all $x\in\Sf$.
\item[(b)] The mass balance condition holds, $\int_{\Sf} f_1(x)\,dx = \int_{\Sf} f_2(y)\,dy$.
\item[(c)] The cost function is $c(x,y) = \frac{1}{2} d_{\Sf}(x,y)^2$.
\item[(d)] The data satisfies the regularity requirement $f_1 \in L^p(\Sf)$ for some $p \geq 1$.
\end{enumerate}
\end{hypothesis}

The first set of hypotheses leads to smooth solutions.  The second set of hypotheses relaxes the assumptions on the data to permit non-smooth solutions.  While this is valid for both cost functions, we consider this relaxation only in the case of the squared geodesic cost.  The Lipschitz continuity of this cost function will allow us to adapt our convergence framework to the non-smooth setting.  In particular, the following regularity results are adapted from Loeper~\cite{Loeper_OTonSphere}.

\begin{theorem}[Regularity]\label{thm:regularity}
The optimal transport problem~\eqref{eq:OT} with data satisfying Hypothesis~\ref{hyp:Smooth} has a solution $u\in C^3(\Sf)$.  The optimal transport problem~\eqref{eq:OT} with data satisfying Hypothesis~\ref{hyp:Nonsmooth} has a solution $u\in C^1(\Sf)$.
\end{theorem}

See Appendix~\ref{app:regularity} for more details on this result.


The solution to~\eqref{eq:PDE1} is unique only up to additive constants.  In order to select the unique mean-zero solution, we add the additional constraint
\bq\label{eq:meanZero}
\langle u \rangle = 0.
\eq
where $\langle \cdot \rangle$ denotes the average of $u$ over $\Sf$.

While this problem can be interpreted classically under fairly general assumptions, for very general density functions ($f_1, f_2 \in L^p(\Sf)$) or for more general manifolds (including smooth compact manifolds such as certain ellipsoids~\cite{figalli} even with $f_1, f_2 \in C^{\infty}(\Sf)$), $C^2$ solutions $u$ need not exist.  Moreover, the type of convergence analysis frequently used for classical solutions of linear equations is not easily adapted to constrained fully nonlinear equations.  For these reasons, it is also advantageous to be able to interpret the system~\eqref{eq:mapConditionSphere}-\eqref{eq:meanZero} in a weak (viscosity) sense.  

To define these weak solutions, we introduce the notation $\Ef(F)$ to denote the space of functions on which the PDE operator~$F$ is elliptic.  We also require the concepts of upper and lower envelopes of a function.

%

\begin{definition}[Semi-continuous envelopes]\label{def:envelopes}
The \emph{upper and lower semicontinuous envelopes} of a function $u$ are given by
\[ u^*(x) = \limsup\limits_{y\to x} u(y), \quad u_*(x) = \liminf\limits_{y\to x} u(y). \]
\end{definition}

\begin{definition}[Viscosity Solutions]\label{def:viscosity}
An upper (lower) semicontinuous function $u: \Sf \rightarrow \mathbb{R}$ is a \emph{viscosity sub (super)-solution } of the PDE~\eqref{eq:PDE1} if for every $x_0 \in \Sf$ and $\phi \in C^{\infty}(\Sf) \cap\Ef(F)$ such that $u - \phi$ has a local maximum (minimum) at $x_0$ we have
\[
F_*^{(*)}(x_0, \phi(x_0), \nabla_{\Sf}\phi(x_0), D^2_{\Sf} \phi(x_0)) \leq (\geq) 0.
\]
%

A continuous function $u:\Omega\to\R$ is a \emph{viscosity solution} of~\eqref{eq:PDE1} if it is both a sub-solution and a super-solution. 
\end{definition}


\subsection{Numerical methods for fully nonlinear elliptic equations}

In order to build convergent methods for \MA type equations on the sphere, we wish to build upon recent developments in the approximation of fully nonlinear elliptic equations.

A powerful contribution to the numerical approximation of elliptic equations was provided by the Barles-Souganidis framework, which states that the solution to a scheme that is consistent, monotone, and $L^{\infty}$-stable will converge to the viscosity solution, provided the underlying PDE satisfies a comparison principle~\cite{BSnum}. The original paper demonstrates the convergence framework posed on an open set $\Omega \subset \mathbb{R}^n$. In our convergence proof, this approach will be naturally adapted to $\Sf$.

In this article, we consider finite difference schemes that have the form
\bq\label{eq:approx1} F^h \left( x,u(x),u(x)-u(\cdot) \right) = 0 \quad x\in \G^h \eq
and
\bq\label{eq:h} h = \sup\limits_{x\in\Omega}\min\limits_{y\in\G^h}\norm{ x-y } \eq
denotes the grid resolution.

In this setting, the properties required by the Barles-Souganidis framework can be defined as follows.  Consider the PDE
\bq\label{eq:PDE} F(x,\nabla\phi(x),D^2\phi(x)) = 0, \ x \in \Omega \eq

\begin{definition}[Consistency]\label{def:consistency}
The scheme~\eqref{eq:approx1} is \emph{consistent} with the PDE~\eqref{eq:PDE}
 if for any smooth function $\phi$ and $x\in\bar{\Omega}$,
\[ \limsup_{h\to0,y\to x, z\in\G^h\to x,\xi\to0} F^h(z,\phi(y)+\xi,\phi(y)-\phi(\cdot)) \leq F^*(x,\phi(x),\nabla\phi(x),D^2\phi(x)), 
\]
\[ \liminf_{h\to0,y\to x, z\in\G^h\to x,\xi\to0} F^h(z,\phi(y)+\xi,\phi(y)-\phi(\cdot)) \geq F_*(x,\phi(x),\nabla\phi(x),D^2\phi(x)). \]
\end{definition}

To consistent schemes, we also associate a truncation (consistency) error $\tau(h)$ .
\begin{definition}[Truncation error]\label{def:truncation}
The truncation error $\tau(h) > 0$ of the scheme~\eqref{eq:approx1} is a quantity chosen so that for every smooth function $\phi$
\[ \limsup\limits_{h\to0}\max\limits_{x\in\G^h}\frac{\abs{F^h(x,\phi(x),\phi(x)-\phi(\cdot)) - F(x,\nabla\phi(x),D^2\phi(x))}}{\tau(h)} < \infty. \]
\end{definition}

\begin{definition}[Monotonicity]\label{def:monotonicity}
The scheme~\eqref{eq:approx1} is \emph{monotone} if $F^h$ is a non-decreasing function of its final two arguments.
\end{definition}

\begin{definition}[Proper]\label{def:proper}
The scheme~\eqref{eq:approx1} is \emph{proper} if $F^h$ is an increasing function of its second argument.
\end{definition}

\begin{definition}[Stability]\label{def:stability}
The scheme~\eqref{eq:approx1} is \emph{stable} if there exists $M\in\R$ (independent of $h$) such that whenever $u^h$ is a solution of~\eqref{eq:approx1} then $\|u^h\|_\infty \leq M$.
\end{definition}

This convergence framework does not apply to all elliptic PDEs, including~\eqref{eq:PDE1}, which does not have the required comparison principle.  Nevertheless, it provides an important starting point for the development of convergent numerical methods.  In particular, monotone schemes possess a weak form of a discrete comparison principle even if the limiting PDE does not~\cite[Lemma~5.4]{hamfeldt2}.  If the scheme additionally exhibits an increasing dependence on the function $u$ itself, we obtain a traditional strong form of the discrete comparison principle that guarantees solution uniqueness.

%

\begin{lem}[Discrete comparison principle~{\cite[Theorem~5]{ObermanSINUM}}]\label{lem:discreteComp}
Let $F^h$ be a monotone, proper scheme and $F^h(x,u(x),u(x)-u(\cdot)) \leq F^h(x,v(x),v(x)-v(\cdot))$ for every $x\in\G^h$.  Then $u(x) \leq v(x)$ for every $x\in\G^h$.
\end{lem}

Another property that has recently proved important in establishing convergence of some numerical methods for the \MA equation is the concept of underestimation~\cite{BenamouDuval,HamfeldtBVP2,LindseyRubinstein}.  This concept will be important for our efforts to extend our convergence framework to the non-smooth setting.

\begin{definition}[Underestimation]\label{def:underest}
The scheme~\eqref{eq:approx1} \emph{underestimates} the PDE~\eqref{eq:PDE} if
\[  F^h(x,u(x),u(x)-u(\cdot)) \leq 0\]
for every (possibly non-smooth) solution $u$ of~\eqref{eq:PDE}.
\end{definition}

One of the biggest challenges in setting up finite difference schemes for fully nonlinear elliptic PDE is satisfying the monotonicity property. Even for some linear elliptic equations, it is not possible to build a consistent, monotone scheme on a finite stencil~\cite{Kocan}. To resolve this issue, wide-stencil schemes have been introduced for a range of fully nonlinear elliptic PDE.  To achieve both consistency and monotonicity, these schemes require the width of finite difference stencils to become unbounded as the grid is refined. A variety of monotone schemes now exist for the \MA equation~\cite{benamou2014monotone,BenamouDuval,FinlayOberman,FO_MATheory,ObermanEigenvalues}, including schemes that can be posed on very general grids~\cite{FroeseMeshfreeEigs,HS_Quadtree,Nochetto_MAConverge}.  With some modification, these methods can be adapted to fit within the convergence framework developed in this article.

\section{PDE on the Sphere}\label{sec:PDE}
We begin by introducing an appropriate characterization of the PDE~\eqref{eq:PDE1}-\eqref{eq:cconvex} on the sphere, which will show how the numerical computations can be performed in local tangent planes.  We also introduce a modification of the PDE that will allow us to build $c$-convexity and additional Lipschitz stability into our numerical framework.

\subsection{Interpretation of the PDE}
Solving the problem~\eqref{eq:mapConditionSphere}-\eqref{eq:meanZero} requires us to interpret averaging operator, gradient, Hessian, and $c$-consistency constraint on the sphere.

The averaging operator is given in the typical way by
\bq\label{eq:average} \langle u \rangle \equiv \frac{\int_{\Sf} u dV}{\int_{\Sf} dV}. \eq

With both cost functions, the gradient (an object in the tangent plane) appears in the mapping $T$. Letting $g$ be the standard round metric on the sphere, then the gradient is given by $\nabla u(x) = g^{ij} \partial_iu \partial_j$, where $\partial_j \in \Tf_x$ and $g^{ij}$ is the inverse of the round metric tensor expressed in local coordinates. The mapping $T$ then can be computed directly by solving~\eqref{eq:mapConditionSphere}.

For the squared geodesic cost, the optimal mapping $T(x,p)$ has a very simple expression in terms of the exponential map. Given a tangent vector $p$ (which, in particular, would include the gradient defined above) the exponential map is defined as
\begin{equation}\label{eq:expMap}
\text{exp}_x(p) = \gamma_{x, p}(\norm{ p }).
\end{equation}
Here $\gamma_{x, p}(t)$ denotes the point a distance $t$ (parametrized by arclength) along the geodesic beginning from $x \in \Sf$ and oriented in the direction $p$. 
Then the optimal map corresponding to the squared geodesic cost is given by:
\[ T(\nabla u(x)) = \text{exp}_x(\nabla u(x)). \]

As in~\cite{McRae_OTonSphere}, this map can be found explicitly as
\bq\label{eq:mapSphere}
T(x,p) = \cos\left(\frac{\norm{ p }}{2}\right) x + \sin\left(\frac{\norm{ p }}{2}\right)\frac{p}{ \norm{ p }}. \vspace*{-4pt}
\eq

We derive a similar explicit form of the optimal map corresponding to the log cost (see Appendix~\ref{app:logCost}):
\bq\label{eq:mapLog}
T(x,p) = x\frac{\norm{ p }^2-1/4}{\norm{ p }^2+1/4}-\frac{p}{\norm{ p }^2+1/4}.
\eq

The explicit formulas for the mapping $T$ for both costs demonstrates that they are continuous functions of the gradient. Thus, a smooth gradient $\nabla u(x)$ leads to a smooth mapping $T$, which simplifies the task of obtaining consistent approximations of the mapping.

Computing derivatives of order $n \geq 2$ in the tangent plane introduces some local distortion due to the choice of coordinate system.  The Hessian on manifolds usually includes an additional first-order term that is non-zero if the Christoffel symbols are non-zero. In our approach in this article, we will be interested in a choice of local coordinates (geodesic normal coordinates) that cause the Christoffel symbols to vanish.  This, in turn, will allow us to compute the spherical Hessian as a ``flat'' Hessian on the local tangent plane.


The condition that a solution $u$ must be $c$-convex~\eqref{eq:cconvex} means that $u$ can be characterized as the $c$-transform of some function $\psi$. For symmetric cost functions, we say that the function $u$ is $c$-convex if there exists a function $\psi$ such that
\begin{equation}\label{eq:ctransform}
u(x) = \sup_{y \in \Sf} \{ -c(x,y) - \psi(y) \} \equiv \psi^c(x).
\end{equation}
For $u$ and $T(x,p)$ smooth and $c$-convex, this condition implies that
\begin{equation}\label{eq:cConvexSmooth}
D^2 u(x) + D^2_{xx} c(x,T(x,\nabla u(x))) \geq 0
\end{equation}
where the inequality here means that the matrix is positive semidefinite.  We remark that the PDE~\eqref{eq:OTPDE} is elliptic only on the space of functions satisfying this constraint.  That is,
\[ \Ef(F) = \{u\in C^2(\Sf) \mid D^2 u(x) + D^2_{xx} c(x,T(x,\nabla u(x))) \geq 0
.\} \] 
%

\subsection{Tangent plane characterization}
In order to actually approximate the PDE~\eqref{eq:OTPDE} at a point $x_0\in\Sf$, we wish to define a set of local coordinates $v_{x_0}(x)$ that will map points on the sphere to points on the tangent plane $\Tf_{x_{0}}$.  This would then allow us to draw from the discretization schemes that are already available for approximating fully nonlinear elliptic PDE in $\R^2$.

We mention that the determinant of the Hessian, and the magnitude and direction of the gradient, are coordinate-invariant quantities.  Our particular choice of normal coordinates is motivated primarily by the desire for computational ease.
We reemphasize that the computational challenge here is that local coordinates can distort the Hessian and require the introduction of an additional first-order term.  To avoid the need to modify the PDE, we choose to work with geodesic normal coordinates.  These retain sufficient local structure of the manifold to cause the Christoffel symbols to vanish, which in turn causes the first-order correction term to vanish.

In particular, this choice of normal coordinates preserves distances from the reference point $x_0$.  That is, if $x\in\Sf$ and $v_{x_0}(x)\in\Tf_{x_{0}}$  are sufficiently close to $x_0$, then
\[ \norm{ x_0-v_{x_0}(x) } = d_{\Sf}(x_0,x). \]
These coordinates also preserve orientation so that the projection of $x-x_0$ into the tangent plane is parallel to $v_{x_0}(x)-x_0$.  On the sphere it is possible to construct such coordinates for neighborhoods of uniform size and, in addition, the mapping $v_{x_0}$ is invertible and differentiable. We compute the following explicit representation in Appendix~\ref{app:normalCoords}:
\bq\label{eq:normalCoords}
v_{x_0}(x) = x_0\left(1-d_{\Sf}(x_0,x)\cot d_{\Sf}(x_0,x)\right) + x \left(d_{\Sf}(x_0,x)\csc d_{\Sf}(x_0,x)\right).
\eq

For each point $x_0\in \Sf$ we can now define a function $\tilde{u}_{x_0}(z)$ on the relevant tangent plane $\Tf_{x_{0}}$ in a neighbourhood of $x_0$ by
\bq\label{eq:tangentFunction}
\tilde{u}_{x_0}(z) = u(v_{x_0}^{-1}(z)).
\eq
This choice of coordinates allows us to express the PDE~\eqref{eq:OTPDE} at the point $x_0\in\Sf$ as a generalized \MA equation
\bq\label{eq:MATangent}
F(x_0,\nabla\tilde{u}(x_0),D^2\tilde{u}(x_0)) \equiv-\det(D^2\tilde{u}(x_0)+A(x_0,\nabla\tilde u(x_0))) + H(x_0,\nabla\tilde u(x_0)) = 0, 
\eq
which is now conveniently posed locally on two-dimensional planes. Thus the problem of approximating the PDE at $x_0$ reduces to the problem of constructing an approximation to the two-dimensional generalized \MA equation~\eqref{eq:MATangent} at $x_0$, posed on the tangent plane containing the points $v_{x_0}(x)$.

We emphasize again that the gradient and Hessian of $\tilde{u}$ on the tangent plane at $x_0$ are equivalent to the surface gradient and Hessian on the original function $u$ on the sphere at $x_0$~\cite[Lemma~4.8 and Proposition~5.11]{LeeManifolds}.  Thus using these local coordinates indeed allows us to interpret our PDE, without modification, on the tangent plane.
\begin{lemma}
Let $u \in C^2(\Sf)$ and $x_0\in\Sf$, with $\tilde{u}:\Tf_{x_{0}}\in\R$ defined in geodesic normal coordinates via~\eqref{eq:tangentFunction}.  Then the PDE operator~\eqref{eq:OTPDE} applied to $u$ at the point $x_0$ is equivalent to the generalized \MA operator~\eqref{eq:MATangent} applied to $\tilde{u}$ at the point $x_0$:
\[ F(x_0,\nabla_{\Sf} u(x_0), D^2_{\Sf} u(x_0)) = F(x_0,\nabla\tilde{u}(x_0),D^2\tilde{u}(x_0)). \]
\end{lemma}

\subsection{Constraints}\label{sec:gradConstraints}
We now turn our attention to the problem of incorporating constraints into the PDE.  We recall that the PDE operator~\eqref{eq:OTPDE} is elliptic only on the space of functions satisfying the constraint~\eqref{eq:cconvex}.  Consequently, this constraint is necessary for the equation to be well-posed.  We propose instead to produce a globally elliptic extension of~\eqref{eq:OTPDE} that does not require additional constraints.  To do so, we introduce a modified determinant operator satisfying
\begin{equation}\label{eq:detPlus}
\text{det}^{+}(M) = 
\begin{cases}
\text{det}(M), \ \ \ M \geq 0 \\
<0, \ \ \ \text{otherwise}.
\end{cases}
\end{equation}
Then we can absorb the constraint into the PDE~\eqref{eq:MATangent} through the modification
\bq\label{eq:MAPlus}
F^+(x,\nabla u(x),D^2u(x)) \equiv-{\det}^+(D^2{u}(x)+A(x,\nabla u(x))) + H(x,\nabla u(x)) = 0.
\eq
Since the function $H > 0$, (sub)solutions of this will automatically satisfy the condition
\[ D^2{u}(x)+A(x,\nabla u(x)) \geq 0. \]

The solution $u$ of~\eqref{eq:mapConditionSphere}-\eqref{eq:meanZero} is also known to satisfy \emph{a priori} bounds on its gradient,
\bq\label{eq:gradBound} \norm{ \nabla u } \leq R \eq
for any $R>\pi$ in the case of the squared geodesic cost and $R>C$ in the case of the logarithmic cost.  Here $C$ is the bound on $\nabla u$ determined in~\cite[Proposition~6.1]{Loeper_OTonSphere}. 

With the goal of constructing Lipschitz stable approximation schemes, we state a modification of the PDE that explicitly includes these constraints on the gradient. 
\begin{equation}\label{eq:modifiedPDE}
G(x,\nabla u(x),D^2u(x)) \equiv \max \left\{ F^{+}(x,\nabla u(x),D^2u(x)), \norm{ \nabla u(x) } - R \right\} = 0.
\end{equation}
We again emphasize that this new PDE is elliptic on all $C^2$ functions ($\Ef(G) = C^2(\Sf)$), and does not require any additional constraints.  Moreover, as we demonstrate below, the $c$-convex solution of~\eqref{eq:OTPDE} is indeed a solution of this modified equation.

\begin{remark}
Under the assumption that the globally elliptic equation~\eqref{eq:modifiedPDE} has a unique solution, it must automatically coincide with the $c$-convex solution of the original equation.
Comparison principles and uniqueness results for many fully nonlinear elliptic PDEs of this form are available~\cite{CIL}.  However, these calculations are highly technical and need to be specifically adapted to the PDE at hand.  This is beyond the scope of the present article. 
\end{remark}

It is not \emph{a priori} obvious that solutions of this new PDE operator will automatically satisfy the original PDE.  Indeed, because of the action of the maximum operator, they need only be subsolutions.
To establish the plausibility of this new operator, we establish that the equivalence of these two equations for smooth, $c$-convex functions.

\begin{theorem}[Equivalence of PDE (smooth case)]\label{thm:equivalenceSmooth}
Under the conditions of Hypothesis~\ref{hyp:Smooth}, a $c$-convex function $u\in C^2$ is a solution of~\eqref{eq:OTPDE} if and only if it is a solution of~\eqref{eq:modifiedPDE}.
\end{theorem}

Before completing the proof, we establish a few lemmas relating to the transportation of mass by subsolutions.  The following proofs will make use of an abbreviated notation for the transport map:
\[ T_u(x) = T(x,\nabla u(x)). \]

%
\begin{lemma}\label{lem:cConvexSmooth}
If  $u\in C^2$ is $c$-convex then it satisfies the constraint~\eqref{eq:cConvexSmooth}: $D^2u(x)+D^2_{xx}c(x,T_u(x)) \geq 0$.
\end{lemma}
\begin{proof}
If $u$ is $c$-convex, then for every $x_0\in\Sf$ we can fix $y = T_u(x_0)$ and find that the supremum in
\[ u^c(y) = \sup\limits_{x\in\Sf}\{-c(x,y)-u(x)\} \]
is attained at $x_0$.  The optimality condition for this is precisely~\eqref{eq:cConvexSmooth}.
\end{proof}

\begin{lemma}\label{lem:transportSmooth}
Under the conditions of Hypothesis~\ref{hyp:Smooth}, let $u\in C^2$ be a subsolution of~\eqref{eq:OTPDE}.   Then
\[ \int_{T_u(\Sf)}f_2(y)\,dy \leq \int_{\Sf} f_1(x)\,dx. \]
\end{lemma}
\begin{proof}
By design, the transport maps~\eqref{eq:mapSphere}-\eqref{eq:mapLog} satisfy $T_u(\Sf) \subset \Sf$.  Because of mass balance we conclude that
\[ \int_{\Sf} f_1(x)\,dx = \int_{\Sf} f_2(y)\,dy \geq \int_{T_u(\Sf)} f_2(y)\,dy. \qedhere \]
\end{proof}

The preceding lemma will be used to derive a contradiction that shows smooth subsolutions of~\eqref{eq:OTPDE} are, in fact, solutions.

\begin{lemma}\label{lem:subSmooth}
Under the conditions of Hypothesis~\ref{hyp:Smooth}, let $u\in C^2(\Sf)$ be a subsolution of~\eqref{eq:OTPDE}.  Then $u$ is a solution of~\eqref{eq:OTPDE}.
\end{lemma}

\begin{proof}
Suppose $u$ is not a solution.  Since $u\in C^2(\Sf)$, there exists some open set $E\subset\Sf$ such that
\[ F(x,\nabla u(x), D^2u(x)) < 0. \]

We recall that the mapping $T_u$ satisfies the condition~\eqref{eq:mapConditionSphere}:
\[
\nabla u(x) =- \nabla_{x} c(x,T_{u}(x)).
\]
Differentiating yields
\[
D^2 u(x) = -D^2_{xx} c(x,T_{u}(x)) - D^2_{xy}c(x,T_{u}(x))DT_{u}(x).
\]

Since $u$ is a subsolution of~\eqref{eq:OTPDE}, we know that
\begin{align*}
\abs{\det (D^2_{xy}c(x,T_u(x)))}f_1(x)/f_2(T_u(x)) & \leq \det(D^2u(x)+D^2_{xx} c(x,T_{u}(x)))\\
  &= \abs{\det (D^2_{xy}c(x,T_{u}(x)))} \det(DT_u(x)).
\end{align*}
Therefore
\[ f_1(x) \leq \det(DT_u(x))f_2(T_u(x)) \]
with strict inequality on an open set $E\subset\Sf$.

Integrating, we obtain
\[ \int_{\Sf} f_1(x)\,dx < \int_{T_u(\Sf)} f_2(y)\,dy. \]
This contradicts Lemma~\ref{lem:transportSmooth} and thus $u$ is a solution of~\eqref{eq:OTPDE}.
\end{proof}

%
%
%
%
%
%
%

\begin{proof}[Proof of Theorem~\ref{thm:equivalenceSmooth}]
Let $u$ be a c-convex solution of~\eqref{eq:OTPDE}.  Then it satisfies the gradient bound $\norm{ \nabla u } - R \leq 0$ from~\eqref{eq:gradBound}.  Because it is c-convex, it also satisfies the constraint~\eqref{eq:cConvexSmooth} (Lemma~\ref{lem:cConvexSmooth}) so that
\[ F^+(x,\nabla u(x), D^2u(x)) = F(x,\nabla u(x), D^2u(x)) = 0. \]
Then trivially the maximum of these operators also vanishes, and the modified PDE~\eqref{eq:modifiedPDE} is satisfied.

Now we let $u$ be a solution of the modified PDE~\eqref{eq:modifiedPDE} so that
\[ \max\left\{F^+(x,\nabla u(x),D^2u(x)), \norm{ \nabla u(x) } - R\right\} = 0. \]
This implies that $u$ is a subsolution of the convexified PDE operator~\eqref{eq:MAPlus} denoted by $F^+$.  Subsolutions of this equation automatically satisfy the constraint~\eqref{eq:cConvexSmooth} (see the definition of $\text{det}^+$) so that
\[ F(x,\nabla u(x), D^2u(x)) = F^+(x,\nabla u(x), D^2u(x)) \leq 0. \]
From Lemma~\ref{lem:subSmooth}, $u$ is necessarily a solution of~\eqref{eq:OTPDE}.
\end{proof}

We also partially extend this equivalence result to the non-smooth case for the squared geodesic cost.
\begin{theorem}[Equivalence of PDE (non-smooth case)]\label{thm:equivalenceNonsmooth}
Under the conditions of Hypothesis~\ref{hyp:Nonsmooth}, let $u\in C^{0,1}(\Sf)$ be a $c$-convex viscosity solution of~\eqref{eq:OTPDE}.  Then $u$ is a viscosity solution of~\eqref{eq:modifiedPDE}.
\end{theorem}

\begin{remark}
The key to proving this result is the observation that subsolutions of the modified equation satisfy \emph{a priori} Lipschitz bounds.  This is fairly straightforward for the squared geodesic cost, but more challenging for the logarithmic cost because of the singularity in the cost function.  A possibility for extending this theorem to singular cost functions, which is explored in~\cite{HT_SphereNumerics}, is to use regularity results to study optimal transportation with an alternative (regularized) version of the logarithmic cost function that yields the same solution as the unregularized problem.
\end{remark}

Once again, we begin with a few lemmas.
\begin{lemma}[Local $c$-convexity of test functions]\label{lem:testFunctions}
Let $u\in C^{0,1}(\Sf)$ be $c$-convex with cost function $c(x,y) = \frac{1}{2}d_{\Sf}(x,y)^2$ and $\phi\in C^\infty(\Sf)$.  Suppose that $u-\phi$ has a local maximum at $x_0$.  Then
\[ D^2\phi(x_0) + D^2_{xx} c(x_0,T_\phi(x_0)) \geq 0. \]
\end{lemma}
\begin{proof}
At the maximizer $x_0$ of $u-\phi$, we must have $\nabla\phi(x_0)\subset\partial u(x_0)$.

Since $u$ is $c$-convex, there exists a function $u^c$ such that
\[ u(x)+u^c(y) = -c(x,y), \quad y \in \partial u(x). \]
Thus the maximizer $x_0$ of $u-\phi$ will also maximize the function $-u^c(y)-c(x,y)-\phi(x)$, where we can in particular choose $y = T_\phi(x_0)$.
The optimality condition for this is
\[ -D_{xx}^2c(x_0,y)-D^2\phi(x_0) \leq 0, \quad y = T_\phi(x_0). \qedhere\]
\end{proof}

\begin{lemma}[Lipschitz bounds on subsolutions]\label{lem:gradSubs}
Let $u\in C^{0,1}$ be $c$-convex where $c(x,y) = \frac{1}{2}d^2_{\mathbb{S}^2}(x,y)$.  Then the Lipschitz constant of $u$ is bounded by $\pi$.
\end{lemma}

\begin{proof}
We first consider $x\in\Sf$ such that $u$ is differentiable at $x$.
As in~\cite{Loeper_OTonSphere}, we define the set
\[
G_{u}(x) = \{y \in \mathbb{S}^2, u(x) + u^{c}(y) = -c(x,y) \}.
\]
Letting $\partial^cu(x)$ denote the $c$-subdifferential of $u$, defined as
\[
\partial^{c} u(x) = \left\{ -\nabla_{x} c(x,y), y \in G_{u}(x) \right\},
\]
due to Loeper~\cite{LoeperReg} Proposition 2.11 we know that for all $c$-convex $u$, 
\[\emptyset \neq \partial^{c} u(x) = \partial u(x) \] 
Thus $\nabla u(x) = \partial^{c} u(x)$. 

To bound $\nabla u$, we need only bound the gradient of the cost function $c(x,y)$:
\[
\nabla_{x} c(x,y) = d_{\mathbb{S}^2}(x,y) \nabla_{x} d_{\mathbb{S}^2}(x,y)
\]
Letting $\hat{n}$ denote a unit tangent vector in the tangent plane $\Tf(x)$, we compute
\[
 \nabla_{x} d_{\mathbb{S}^2}(x,y) \cdot \hat{n} = \lim_{s \rightarrow 0} \frac{d_{\mathbb{S}^2}(\text{exp}_{x}(s \hat{n}),y) - d_{\mathbb{S}^2}(x,y)}{s}.
\]
From the triangle inequality we obtain the bounds
\[
\nabla_{x} \cdot \hat{n} d_{\mathbb{S}^2}(x,y) \leq \lim_{\norm{ \Delta x } \rightarrow 0} \frac{d_{\mathbb{S}^2}(\text{exp}_{x}(s\hat{n}),x) + d_{\mathbb{S}^2}(x,y) - d_{\mathbb{S}^2}(x,y)}{s} = 1
\]
and
\[
\nabla_{x} \cdot \hat{n} d_{\mathbb{S}^2}(x,y) \geq \lim_{s \rightarrow 0} \frac{d_{\mathbb{S}^2}(\text{exp}_{x}(s \hat{n}_x),y) - d_{\mathbb{S}^2}(\text{exp}_{x}(s \hat{n}_x),y) - d_{\mathbb{S}^2}(\text{exp}_{x}(s \hat{n}_x),x)}{s} = -1.
\]

Therefore
\begin{equation}
\norm{\nabla u(x)} \leq \norm{\nabla_{x} c(x,y)} \leq d_{\mathbb{S}^2}(x,y) \leq \pi
\end{equation}
at points $x$ where $u$ is differentiable.  Since $u$ is Lipschitz continuous, this gradient bound is also a bound on the Lipschitz constant.
\end{proof}

\begin{proof}[Proof of Theorem~\ref{thm:equivalenceNonsmooth}]
Suppose that $u$ is a $c$-convex viscosity solution of~\eqref{eq:OTPDE}.  Consider any $x_0\in\Sf$ and $\phi\in C^\infty(\Sf)$ such that $u-\phi$ has a local maximum at $x_0$.  Then
\[ F(x_0,\nabla\phi(x_0),D^2\phi(x_0)) \leq 0. \]
Moreover, since $u-\phi$ is a maximum we know that $\nabla\phi(x_0)\subset\partial u(x_0)$. From Lemma~\ref{lem:gradSubs} we find that $\norm{\nabla\phi(x_0)} - R < 0$. {Additionally, since $u$ is $c$-convex, $\phi$ must be locally $c$-convex as well near $x_0$ (Lemma~\ref{lem:testFunctions}) so that $\phi\in\Ef(F)$ is a valid test function for the original PDE operator.}    Thus
\[ F^+(x_0,\nabla\phi(x_0),D^2\phi(x_0)) = F(x_0,\nabla\phi(x_0),D^2\phi(x_0)) \leq 0 \]
and the modified operator will satisfy
\[ \max\{F^+(x_0,\nabla\phi(x_0),D^2\phi(x_0)),\norm{\nabla\phi(x_0)} - R\} \leq 0. \]
Therefore $u$ is a sub-solution of~\eqref{eq:modifiedPDE}.

Next we consider $x_0\in\Sf$ and $\phi\in C^\infty(\Sf)$ such that $u-\phi$ has a local minimum at $x_0$.  
If $\phi$ satisfies the constraint~\eqref{eq:cConvexSmooth} then $\phi\in\Ef(F)$ is a valid test function for the original PDE operator.
Thus, by the fact that $u$ is a supersolution of~\eqref{eq:OTPDE}, we have
\[\max\{F^+(x_0,\nabla\phi(x_0),D^2\phi(x_0)),\norm{\nabla\phi(x_0)} - R\} \geq F(x_0,\nabla\phi(x_0),D^2\phi(x_0)) \geq 0.\]
Otherwise, $D^2\phi(x_0)+D^2_{xx}c(x_0,T_\phi(x_0))$ is not positive semi-definite.  From the definition of the modified determinant operator~\eqref{eq:detPlus}, this means that
\[ F^+(x_0,\nabla\phi(x_0),D^2\phi(x_0)) \geq -{\det}^+(D^2\phi(x_0)+D^2_{xx}c(x_0,T_\phi(x_0))) > 0. \]
This again leads to the inequality
\[ \max\{F^+(x_0,\nabla\phi(x_0),D^2\phi(x_0)),\norm{\nabla\phi(x_0)} - R\} > 0.\]
In either case, we conclude that $u$ is a super-solution, and therefore also a viscosity solution, of~\eqref{eq:modifiedPDE}.

\end{proof}

\section{Convergence Framework}\label{sec:numerics}

%

\subsection{Discrete formulation}\label{sec:discrete}
In order to numerically solve~\eqref{eq:OTPDE}, we begin with a point cloud $\G^h \subset\Sf$ that discretizes the sphere. We define the discretization parameter $h$ as
\begin{equation}\label{eq:h}
h = \sup\limits_{x\in\Sf}\min\limits_{y\in\G^h} d_{\Sf}(x,y).
\end{equation}
In particular, this guarantees that any ball of radius $h$ on the sphere will contain at least one discretization point.

We will impose some mild structural regularity on the grid.
\begin{hypothesis}[Conditions on point cloud]\label{hyp:grid}
There exists a triangulation $T^h$ of $\G^h$ with the following properties:
\begin{enumerate}
\item[(a)] The diameter of the triangulation, defined as
\bq\label{eq:diam}  \text{diam}(T^h) = \max\limits_{t\in T^h} \text{diam}(t),\eq
 satisfies $\text{diam}(T^h)\to0$ as $h\to0$.
\item[(b)] There exists some $\gamma < \pi$ (independent of $h$) such that whenever $\theta$ is an interior angle of any triangle $t\in T^h$ then $\theta \leq \gamma$.
\end{enumerate}
\end{hypothesis}
We remark that these are fairly standard assumptions on a grid: we are simply prohibiting long, thin triangles.

We also associate to each point cloud $\G^h$ a search radius $r(h)$ chosen to satisfy
\bq\label{eq:r} r(h)\to0, \, \frac{h}{r(h)} \to 0 \text{ as } h\to 0, \quad \text{diam}(T^h) < r(h). \eq

%
%
%

Now we considering the problem of constructing a discretization of~\eqref{eq:modifiedPDE} at the point $x_0\in\G^h$.  We begin by projecting nearby grid points onto the local tangent plan $\Tf_{x_{0}}$, which is spanned by the orthonormal vectors $\left( \hat{\theta},\hat{\phi} \right)$.  For all points $x_i\in\G^h\cap B(x_0,r(h))$, we define their projection onto the tangent plane through geodesic normal coordinates via
\bq\label{eq:neighbours}
z_i = x_0 \left(1 - d_{\Sf}(x_0,x_i) \cot d_{\Sf}(x_0,x_i) \right) + x_i \left( d_{\Sf}(x_0,x_i) \csc d_{\Sf}(x_0,x_i) \right).
\eq
Let $\Zf^h(x_0)\subset\Tf_{x_{0}}$ be the resulting collection of points.
See Figure~\ref{fig:sphere}.

\begin{figure}[htp]
\centering
\subfigure[]{
\includegraphics[width=0.5\textwidth,clip=true,trim=0 1in 0 0.8in]{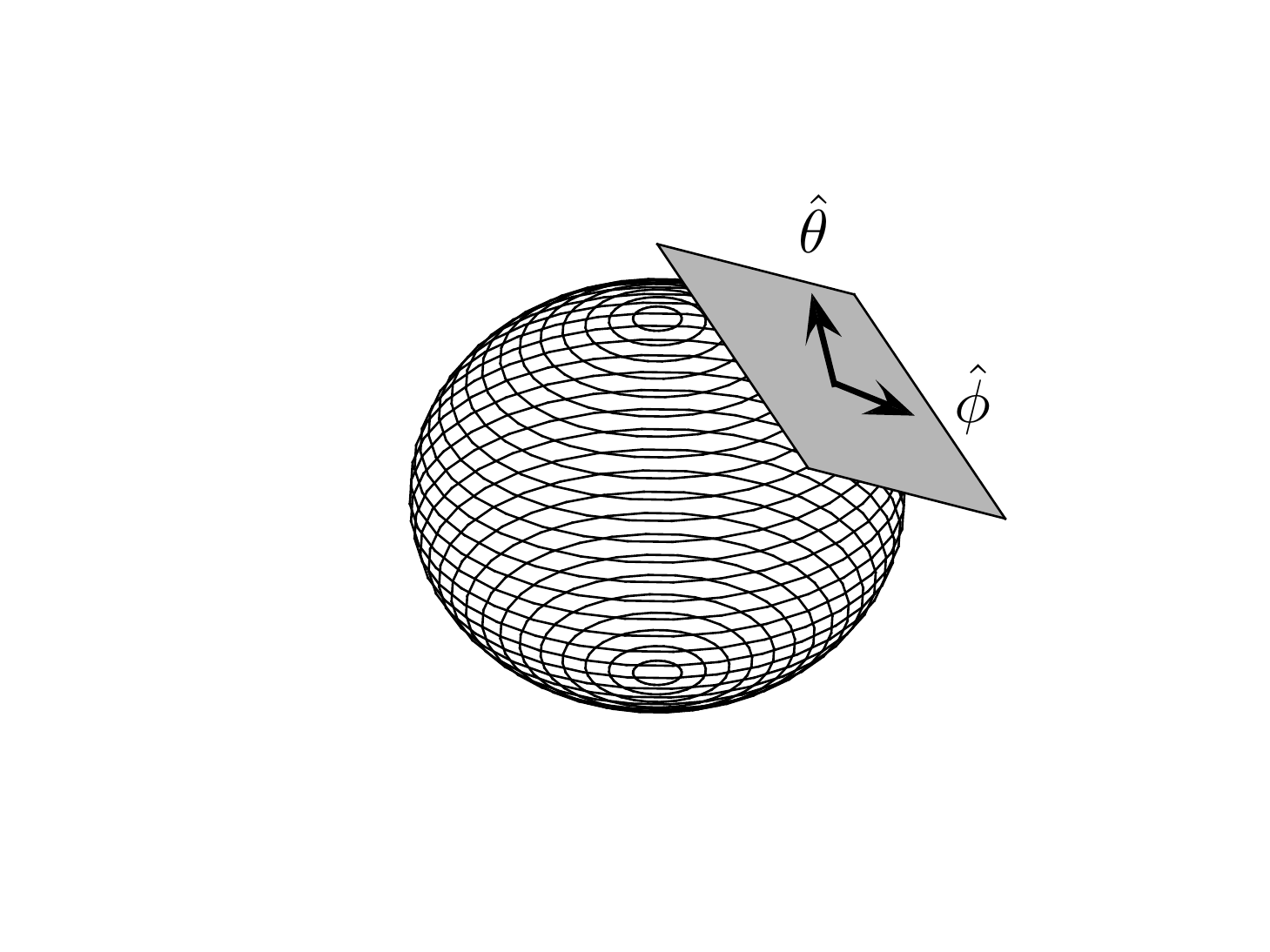}\label{fig:sphere1}} 
\subfigure[]{
\includegraphics[width=0.3\textwidth]{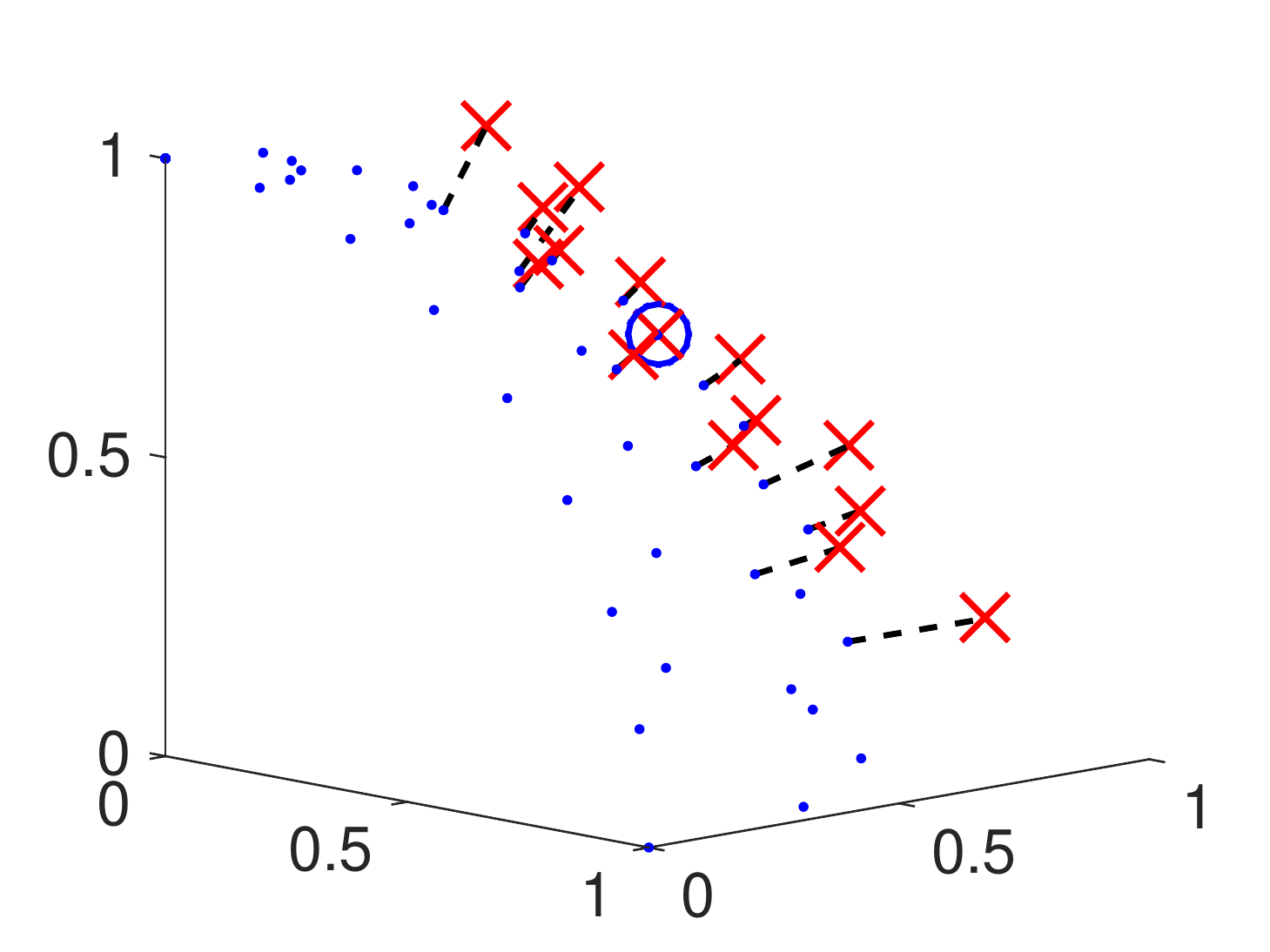}\label{fig:sphere2}}
\caption{\subref{fig:sphere1}~The sphere $\Sf$ and tangent plane $\Tf_{x_{0}}$. \subref{fig:sphere2}~A point cloud discretizing one octant of the unit sphere ($\cdot$), the point $x_0$ (o), and the projections $z$ of neighboring nodes onto $\Tf_{x_{0}}$ ($\times$). }
\label{fig:sphere}
\end{figure}

These are now the discretization points available to use for the approximation of~\eqref{eq:modifiedPDE} at $x_0$; recall that this PDE is posed on the two-dimensional tangent plane.  There are three components to this discretization: approximation of the \MA type operator $F^+(z,\nabla u(z),D^2u(z))$~\eqref{eq:MAPlus}, approximation of the Eikonal term $\nabla u(x)$, and approximation of the averaging term~$\langle u \rangle$. 
Let $F^h$, $E^h$, and $A^h$ be suitable discretizations of these three operators. 

Our framework will allow for a very general choice of schemes $F^h$ and $A^h$.  In particular, many currently available methods for the \MA equation can be adapted to fit within our requirements.  The specific requirements are:
\begin{hypothesis}[Conditions on schemes]\label{hyp:schemes}
We require the schemes $F^h(x,u(x)-u(\cdot))$ and $A^h(u(\cdot))$ to satisfy:
\begin{enumerate}
\item[(a)] $F^h$ is consistent with~\eqref{eq:MAPlus} on all $C^2$ smooth functions.
\item[(b)] $F^h$ is monotone.
\item[(c)] $A^h$ is consistent with the averaging operator~\eqref{eq:average} on all Lipschitz continuous functions.
\item[(d)] $A^h$ is linear and $A^h(c) = c$ for any constant function $c$.
\end{enumerate}
\end{hypothesis}

If we wish to obtain non-smooth solutions, $F^h$ will also need to be underestimating. We will require additional structure on $E^h$ in order to obtain the strong form of stability needed to guarantee convergence.  In particular, we propose
\bq\label{eq:eik}
E^h(z,u(z)-u(\cdot)) = \max\limits_{y\in\Zf^h(z)}\frac{ u(z)-u(y) }{\norm{z-y}},
\eq
which is consistent with $\norm{ \nabla u(z) }$ and monotone (Lemma~\ref{lem:eik}).

This allows us to produce the following consistent, monotone approximation of~\eqref{eq:modifiedPDE}:
\bq\label{eq:approx2}
G^h(x,u(x)-u(\cdot)) = \max\{F^h(x,u(x)-u(\cdot)),E^h(x,u(x)-u(\cdot))-R\}.
\eq

Finally, we represent our overall approach through the following two-step approach:
\begin{enumerate}
\item[1.] Solve the discrete system
\bq\label{eq:vhScheme}
G^h(x,v^h(x)-v^h(\cdot))+\tau(h)v^h(x) = 0, \quad x \in \G^h
\eq
for the grid function $v^h$.
\item[2.] Define the candidate solution
\bq\label{eq:uh}
u^h(x) = v^h(x) - A^h(v^h(\cdot)), \quad x \in \G^h.
\eq
\end{enumerate}

We remark that our candidate solution $u^h$ could also be obtained directly through solution of the non-local approximation scheme
\bq\label{eq:uhScheme}
\begin{aligned}
G^h(x,u^h(x)-u^h(\cdot)) +\tau(h)u^h(x) + A^h\left(G^h(x,u^h(x)-u^h(\cdot))\right) = 0.
\end{aligned}
\eq

\subsection{Stability}\label{sec:stability}
We now establish some important stability properties of the solutions~$v^h$, $u^h$ of the schemes~\eqref{eq:vhScheme}-\eqref{eq:uhScheme}.  Consistency and monotonicity underpin these results.  They are built into our hypotheses on the scheme for the \MA type operator in order to allow for great flexibility in the numerical method.  However, we also need to establish these properties for our proposed discretization of the Eikonal operator.
\begin{lemma}[Approximation of Eikonal operator]\label{lem:eik}
The scheme~$E^h$ is consistent with $\norm{ \nabla u }$ and monotone.
\end{lemma}

\begin{proof}
Monotonicity is immediately evident from the definition of $E^h$~\eqref{eq:eik}.

Now we recall that the magnitude of the gradient can be characterized as a maximal directional derivative,
\[ \norm{ \nabla u } = \max\limits_{\norm{ \nu } = 1} \frac{\partial u}{\partial \nu}. \]
We can obtain an approximation of the first directional derivative in the direction $\nu = \dfrac{z-y}{\norm{ z-y }}$ via standard backward differencing:
\bq\label{eq:backward} \Dt_{z-y}u(z) = \frac{u(z)-u(y)}{\norm{z-y}}. \eq

Now we consider the set of all such directions that can be resolved using our given set of neighbours $\Zf^h(z)$, defined as
\[ V^h(z) = \left\{\frac{z-y}{\norm{ z-y }} \mid y \in \Zf^h(z)\right\}. \]
The discretization $E^h$ can be rewritten as
\bq\label{eq:eik2} E^h(z,u(z)-u(\cdot)) = \max\limits_{\nu\in V^h(z)} \Dt_\nu u(z). \eq

We denote the directional resolution of this approximation by $d\theta$, which can be computed by
\[ d\theta = \sup\limits_{\norm{\nu} = 1} \min\limits_{y\in\Zf^h(z)} \cos^{-1}\left(\frac{z-y}{\norm{ z-y }}\cdot\nu\right). \]
We also remark that projecting the points $x_i\in\G^h\cap B(x_0,r(h))$ onto the plane preserves both the spacing of grid points $h$ and the effective search radius $r(h)$ up to a constant scaling.  Since $r(h)\to0$, the effective grid spacing also goes to zero and thus~\eqref{eq:backward} is a consistent differencing operator.  Since $\dfrac{h}{r(h)} \to 0$ as $h\to 0$, we will also have  $d\theta\to0$ as $h\to0$ as in~\cite[Lemma~11]{FroeseMeshfreeEigs}.  Thus $E^h$ defined as~\eqref{eq:eik2} is consistent.
\end{proof}

An immediate consequence of this is the consistency and monotonicity of our overall scheme~\eqref{eq:approx2}.
\begin{lemma}[Consistency and monotonicity]
Let $\G^h$ and $F^h$ satisfy the conditions of Hypotheses~\ref{hyp:grid} and~\ref{hyp:schemes} respectively.  The the approximation $G^h$ given by~\eqref{eq:approx2} is monotone and consistent with the PDE~\eqref{eq:modifiedPDE}.
\end{lemma}

We now use the monotonicity property (and resulting discrete comparison principle) to establish existence and bounds for the solution to our approximation scheme.

\begin{lemma}[Existence and stability (smooth case)]\label{lem:stabilitySmooth}
Consider the schemes~\eqref{eq:vhScheme}-\eqref{eq:uh} under the conditions of Hypothesis~\ref{hyp:Smooth}, \ref{hyp:grid}, and~\ref{hyp:schemes}.  Then solutions $v^h$, $u^h$ exist and are unique.  Moreover, there exists some $M>0$ (independent of $h$) such that $\|v^h\|_\infty, \|u^h\|_\infty \leq M$ for all sufficiently small $h>0$.
\end{lemma}

\begin{proof}
We remark first of all that the scheme~\eqref{eq:vhScheme} is monotone and proper and therefore has a unique solution $v^h$~\cite[Theorem~8]{ObermanSINUM}, which immediately yields existence of $u^h$.

Let $u$ be the unique mean-zero solution to the PDE~\eqref{eq:modifiedPDE}.  We know that $u\in C^3(\Sf)$ (Theorem~\ref{thm:regularity}) and consequently is bounded.
From consistency of the scheme~\eqref{eq:approx2} we have that
\[\abs{G^h(x,u(x)-u(\cdot))} \leq \tau(h)  \]
for all $x\in\G^h$ and sufficiently small $h>0$.

Now we choose some $c>0$ and substitute $u + c$ into the scheme~\eqref{eq:vhScheme}.
\begin{align*}
G^h(&x,(u(x)+ c)-(u(\cdot)+ c)) + \tau(h)(u(x)+ c)  \geq -\tau(h) + \tau(h)(-\|u\|_\infty+c)\\
 &> 0 \\
 &= G^h(x,v^h(x)-v^h(\cdot)) + \tau(h)v^h(x)
\end{align*}
for $c > \|u\|_\infty + 1$.  By the discrete comparison principle (Lemma~\ref{lem:discreteComp}), we have that $v^h \leq u + c \leq 2\|u\|_\infty+1$.  A similar argument produces a lower bound for $v^h$.

This allows us to also bound the discrete average of $v^h$ via
\[ A^h(v^h(\cdot)) \leq A^h(2\|u\|_\infty+1) = 2\|u\|_\infty+1, \]
with a similar lower bound.

Since $v^h$ and $A^h(v^h)$ are bounded uniformly, $u^h = v^h-A^h(v^h)$ is also bounded uniformly.
\end{proof}

With some additional structure on our discretization, we can modify this stability result to also hold in the non-smooth setting.
\begin{lemma}[Existence and stability (non-smooth case)]\label{lem:stabilityNonsmooth}
Consider the schemes~\eqref{eq:vhScheme}-\eqref{eq:uh} under the conditions of Hypothesis~\ref{hyp:Nonsmooth}, \ref{hyp:grid}, and~\ref{hyp:schemes}.  Suppose also that $F^h$ is an underestimating scheme.  Then solutions $v^h$, $u^h$ exist and are unique.  Moreover, there exists some $M>0$ (independent of $h$) such that $\|v^h\|_\infty, \|u^h\|_\infty \leq M$ for all sufficiently small $h>0$.
\end{lemma}
\begin{proof}
As in Lemma~\ref{lem:stabilitySmooth}, $v^h$ and $u^h$ are uniquely defined.

Let $u$ be the exact mean-zero solution of~\eqref{eq:modifiedPDE}.  Now we know that $u$ is Lipschitz continuous with Lipschitz constant less than $R$.  This implies that
\[ E^h(x,u(x)-u(\cdot)) = \max\limits_{y\in\Zf^h(x)}\frac{ u(x)-u(y) }{\norm{ x-y }} \leq R. \]
Because $F^h$ is an underestimating scheme, we also know that
\[ F^h(x,u(x)-u(\cdot)) \leq 0. \]
Choosing any $c>\|u\|_\infty$ we then obtain
\begin{align*}
G^h(&x,(u(x)- c)-(u(\cdot)- c)) + \tau(h)(u(x)- c)  \leq \tau(h)(\|u\|_\infty-c)\\
 &< 0 \\
 &= G^h(x,v^h(x)-v^h(\cdot)) + \tau(h)v^h(x)
\end{align*}
and by the discrete comparison principle we have the bound $v^h \geq u-\|u\|_\infty \geq -2\|u\|_\infty$.

A simple smooth supersolution of the PDE~\eqref{eq:modifiedPDE} is the constant function $\phi(x) = c$.  Substituting this into the consistent scheme we find that
\begin{align*}
G^h(&x,\phi(x)-\phi(\cdot))+\tau(h)\phi(x) \geq -\tau(h) + \tau(h)c\\
  &>0\\
	&= G^h(x,v^h(x)-v^h(\cdot)) + \tau(h)v^h(x)
\end{align*}
if we choose $c>1$, which yields the bound $v^h \leq 1$.

As in Lemma~\ref{lem:stabilitySmooth}, these uniform bounds on $v^h$ immediately yield uniform bounds on $u^h$.
\end{proof}

An immediate consequence of this is that $u^h$ satisfies a discrete system that is consistent with the PDE~\eqref{eq:modifiedPDE}.
\begin{lemma}[Scheme for $u^h$]\label{lem:uhScheme}
Under the hypotheses of either Lemma~\ref{lem:stabilitySmooth} or Lemma~\ref{lem:stabilityNonsmooth}, $u^h$ satisfies a scheme of the form
\bq\label{eq:uhScheme2}
G^h(x,u^h(x)-u^h(\cdot)) + \tau(h)u^h(x) + \sigma(h) = 0
\eq
where $\sigma(h)\to0$ as $h\to0$.
\end{lemma}

Another immediate consequence of these lemmas is that $u^h$ satisfies a discrete Lipschitz bound uniformly in $h$.
\begin{lemma}[Discrete Lipschitz bounds]\label{lem:lipschitzDisc}
Under the hypotheses of either Lemma~\ref{lem:stabilitySmooth} or Lemma~\ref{lem:stabilityNonsmooth}, $u^h$ satisfies a local discrete Lipschitz bound of the form
\bq\label{eq:lipschitzDisc}
\abs{u^h(z)-u^h(y)} \leq L\norm{ z-y }
\eq
for all $y \in \Zf^h(z)$ and sufficiently small $h>0$ where $L\in\R$ is independent of $h$.
\end{lemma}

\begin{proof}
Note that $u^h$ satisfies~\eqref{eq:uhScheme2}.  For small enough $h$, we can assume $\tau(h), \abs{\sigma(h)} < 1$ and $\|u^h\|_\infty \leq M$.
By construction,
\[ E^h(z,u^h(z)-u^h(\cdot)) \leq G^h(z,u^h(z)-u^h(\cdot)) = -\tau(h)u^h(z) - \sigma(h) \leq M+1 \equiv L. \]
From the definition of $E^h$, we then have
\[ u^h(z)-u^h(y) \leq L\norm{z-y}, \quad y\in\Zf^h(z). \]

If $u^h(z)-u^h(y)\geq 0$ we are done.  Otherwise, we notice that $z \in \Zf^h(y)$ and we can use the fact that
\[ 0 < u^h(y)-u^h(z) \leq L\norm{y-z}, \]
which establishes the result.
\end{proof}

Because of our choice of geodesic normal coordinates, we can immediately extend this to a discrete Lipschitz bound for the function $u^h$ defined on $\G^h\subset\Sf$ in terms of geodesic distances on the sphere (rather than distances on the tangent plane).
\begin{lemma}[Discrete Lipschitz bounds on sphere]\label{lem:lipschitzDiscSphere}
Under the hypotheses of either Lemma~\ref{lem:stabilitySmooth} or Lemma~\ref{lem:stabilityNonsmooth}, $u^h$ satisfies a local discrete Lipschitz bound of the form
\bq\label{eq:lipschitzDiscSphere}
\abs{u^h(x)-u^h(y)} \leq L d_{\Sf}(x,y)
\eq
for all $x\in\G^h$, $y\in\G^h\cap B(x,r(h))$, and sufficiently small $h>0$. Here $L\in\R$ is independent of $h$.
\end{lemma}

\subsection{Interpolation}
In order to establish convergence of the grid function $u^h$ to the solution of~\eqref{eq:modifiedPDE}, we will need to construct an appropriate (Lipschitz continuous) extension of it onto the sphere.

We start by considering linear interpolation of a grid function $w:\G^h\to\R$ onto the triangulated surface $T^h$ described in Hypothesis~\ref{hyp:grid}.  In particular, we want to show that the local discrete Lipschitz bounds~\eqref{eq:lipschitzDiscSphere} are inherited by the resulting piecewise linear interpolant.

\begin{lemma}[Interpolation onto triangulated surface]\label{lem:interpTri}
Let $\G^h$ be a point cloud satisfying Hypothesis~\ref{hyp:grid} and let $w:T^h\to\R$ be a piecewise linear function, linear on each triangle $t\in T^h$, that satisfies the local discrete Lipschitz bounds~\eqref{eq:lipschitzDiscSphere}.  Then there exists some $L\in\R$ (independent of $h$) such that for every $t\in T^h$ and $x,y\in T$, $w$ satisfies the Lipschitz bound $\abs{w(x)-w(y)} \leq L\norm{ x-y }$.
\end{lemma}

\begin{proof}
First we consider the gradient of $w$ on a single triangle $t\in T^h$.  Let $t$ have the vertices $x_0, x_1, x_2 \in \G^h$.  Without loss of generality, we suppose that the maximal interior angle of $t$ occurs at the vertex $x_0$.  Since $\text{diam}(T^h) < r(h) \to 0$ as $h\to0$, there exists a constant $\tilde{L}$ (independent of $h$) such that
\[ \abs{w(x_i)-w(x_j)} \leq L d_{\Sf}(x_i,x_j) = 2L\sin^{-1}\left(\frac{\norm{ x_i-x_j }}{2}\right)\leq \tilde{L} \norm{ x_i-x_j } \]
for all $i,j\in\{0,1,2\}$.
That is, we also have discrete Lipschitz bounds on this triangle.  

For $x\in t$, we can express $w$ as
\[ w(x) = w(x_0) + q\cdot(x-x_0) \]
where $q$ is in the space spanned by $x_1-x_0$ and $x_2-x_0$; that is,
\[ q = q_1(x_1-x_0) + q_2(x_2-x_0) \]
for some $q_1, q_2\in\R$.
  We also denote by $\theta$ the angle between $x_1-x_0$ and $x_2-x_0$.  Note that $\theta \leq \gamma < \pi$ under Hypothesis~\ref{hyp:grid}.  
	
Then at the vertices of $t$ we can write
\[ w(x_i) = w(x_0) + q_i\norm{ x_i-x_0 }^2 + q_j\norm{ x_i-x_0 } \norm{ x_j-x_0 } \cos\theta, \quad i,j\in\{1,2\}, i \neq j. \]
Solving this system for the coefficients $q_1, q_2$, we find that
\[ q_1 = \frac{(w(x_2)-w(x_0))\norm{ x_1-x_0 } \cos\theta-(w(x_1)-w(x_0))\norm{ x_2-x_0 }}{\norm{ x_1-x_0 }^2\norm{ x_2-x_0}(\cos^2\theta-1)}. \]
Applying the discrete Lipschitz bound and since $\theta$ is the largest interior angle of the triangle $t$, we have $\dfrac{\pi}{3} \leq \theta \leq \gamma$, so
\begin{align*}
\abs{q_1} & \leq \frac{\tilde{L} \left( \norm{ x_1-x_0 } \norm{ x_2-x_0 } \abs{\cos\theta} + \norm{ x_1-x_0 } \norm{ x_2-x_0 } \abs{\cos\theta} \right)}{\norm{ x_1-x_0}^2\norm{ x_2-x_0}(1-\cos^2\theta)}\\
  &= \frac{\tilde{L}(\cos\theta+1)}{\norm{ x_1-x_0} (1-\cos^2\theta)}\\
	&\leq \frac{\tilde{L}}{\norm{ x_1-x_0 } (1-\cos\gamma)},
\end{align*}
with a similar bound on $q_2$.

Combining these, we find that
\[ \abs{q} \leq \abs{q_1}\norm{ x_1-x_0 } + \abs{q_2} \norm{ x_2-x_0 } \leq \frac{2\tilde{L}}{1-\cos\gamma}. \qedhere\]
%
\end{proof}

In particular, we can define $w^h:T^h\to\R$ as the unique piecewise linear interpolant of $u^h:\G^h\to\R$ that is linear on each triangle $t\in T^h$.  Notice that $w^h$ satisfies the Lipschitz bounds of Lemma~\ref{lem:interpTri}.  This allows us to produce a Lipschitz continuous interpolant of $u^h$ on the sphere by means of the closest point projection $\text{cp}:T^h\to{\Sf}$,
\bq\label{eq:cp}
\text{cp}(x) = \frac{x}{\norm{ x }}.
\eq
We remark that since $\text{diam}(T^h) \to 0$, this is a bijection for small enough $h>0$.  

This leads to the following extension of $u^h$ onto the sphere:
\bq\label{eq:uhSphere}
u^h(x) = w^h(\text{cp}^{-1}(x)).
\eq
That is, each triangle $t\in T^h$ is distorted to a spherical triangle (Figure~\ref{fig:closestpoint}).  Importantly, this does not significantly distort the gradient of the underlying function values, and uniform Lipschitz bounds are preserved.

\begin{figure}[htp]
\centering
\includegraphics[width=0.65\textwidth]{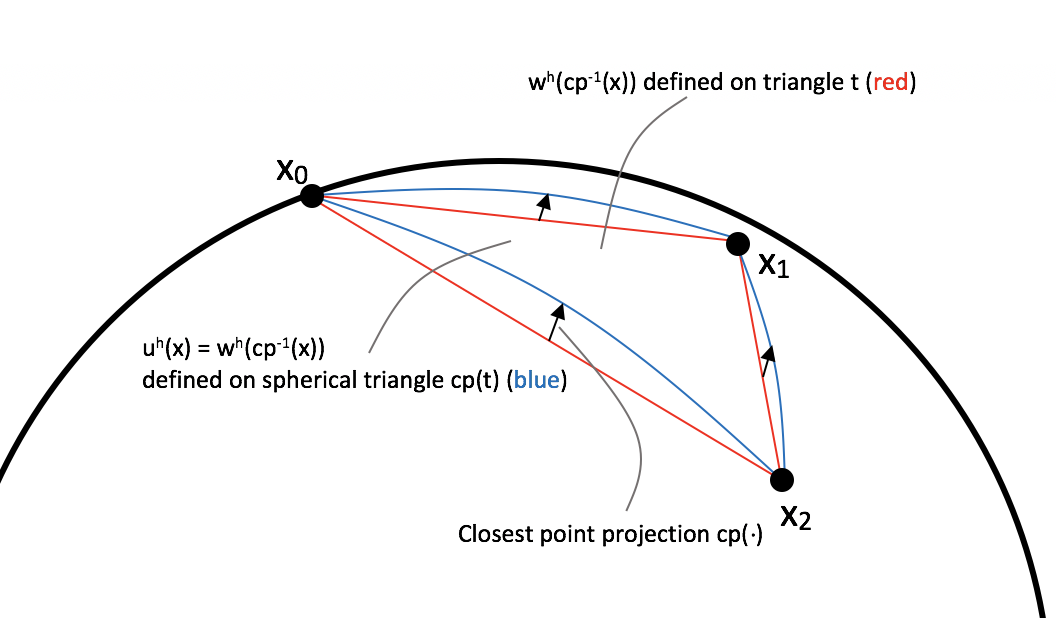}
\caption{Each triangle $t\in T^h$ is distorted via the inverse closest point map to a corresponding spherical triangle.}
\label{fig:closestpoint}
\end{figure}

\begin{lemma}[Lipschitz bounds on the sphere]\label{lem:lipschitzSphere}
Let $u^h:\Sf\to\R$ be as defined in~\eqref{eq:uhSphere}.  Under the hypotheses of either Lemma~\ref{lem:stabilitySmooth} or Lemma~\ref{lem:stabilityNonsmooth}, there exists some $L>0$ (independent of $h$) such that
\[ \abs{u^h(x)-u^h(y)} \leq L d_{\Sf}(x,y) \]
for all $x,y\in\Sf$.
\end{lemma}

\begin{proof}
Let us first consider a fixed triangle $t\in T^h$ and choose any $x,y\in\Sf$ such that $\text{cp}^{-1}(x), \text{cp}^{-1}(y) \in t$.  From Lemma~\ref{lem:interpTri}, we can immediately see that there is some $L>0$ (independent of $h$ and the particular choice of triangle) such that
\bq\label{eq:uhLip} \abs{u^h(x)-u^h(y)} = \abs{w^h(\text{cp}^{-1}(x))-w^h(\cp^{-1}(y))} \leq L\norm{ \cp^{-1}(x)-\cp^{-1}(y) }. \eq

Now we choose a coordinate system such that the triangle $t$ lies in the plane $x_3 = c$.  We recall that $\text{diam}(t) \leq \text{diam}(T^h)\to0$ and the vertices of $t$ lie on the unit sphere $\Sf$.  Thus there exists some $\eta = \bO(\text{diam}(T^h))$ such that 
\bq\label{eq:triValues} \abs{z_1}, \abs{z_2} \leq \eta, \quad 0 \leq 1-z_3 = 1-c \leq \eta \eq
for any $z \in t$.  (See also Figure~\ref{fig:closestpoint}).

In this coordinate system, we can express the closest point function and its inverse as
\[\cp(z) = \frac{(z_1,z_2,c)}{\sqrt{z_1^2+z_2^2+c^2}}, \quad \cp^{-1}(x) = \left(\frac{cx_1}{\sqrt{1-x_1^2-x_2^2}},\frac{cx_2}{\sqrt{1-x_1^2-x_2^2}},c\right).  \]
Notice that we can interpret the first two components of $\cp^{-1}$ as a transformation from $t\in\R^2$ to $\R^2$.  The Jacobian of this transformation is given by
\[ \nabla \tilde{\cp}^{-1}(x) = \left(\begin{tabular}{cc} $c(1-x_2^2)(1-x_1^2-x_2^2)^{-3/2}$ & $cx_1x_2(1-x_1^2-x_2^2)^{-3/2}$ \\ $cx_1x_2(1-x_1^2-x_2^2)^{-3/2}$ & $c(1-x_1^2)(1-x_1^2-x_2^2)^{-3/2}$\end{tabular}\right), \]
which converges uniformly to the identity matrix as $h\to0$ given the estimates on the values of $(x_1,x_2)\in t$ from~\eqref{eq:triValues}.  Thus for sufficiently small $h>0$, we have that $\|\nabla\tilde{\cp}^{-1}(x)\| \leq 2$ for all $(x_1,x_2)\in t$.

This leads to a uniform Lipschitz bound on the inverse closest point map, interpreted as a function on $\R^2$.  For $x,y\in\Sf$ and sufficiently small $h>0$ we then obtain the estimates
\begin{align*}
\norm{ \cp^{-1}(x)-\cp^{-1}(y) } &= \norm{ \tilde{\cp}^{-1}(x)-\tilde{\cp}^{-1}(y) }\\
  &\leq 2\norm{ (x_1,x_2)-(y_1,y_2) }\\
  &\leq 2\norm{ x-y }\\
	&\leq 4d_{\Sf}(x,y).
\end{align*} 
Substituting this into~\eqref{eq:uhLip} yields the desired uniform Lipschitz bounds on any spherical triangle $\cp(t)$.

Since $u^h$ is continuous, its Lipschitz constant is the maximal Lipchitz constant over any spherical triangle, which can be bounded independent of $h$.
\end{proof}

\subsection{Convergence Theorem}
We are now prepared to complete the proof of convergence of the numerical approach outlined in~\autoref{sec:discrete}.
We begin with two lemmas pertaining to uniformly convergent sequences $u^{h_n}$.

\begin{lemma}\label{lem:meanZero}
Let $u^h$ be defined by the schemes~\eqref{eq:vhScheme}-\eqref{eq:uh} and~\eqref{eq:uhSphere} under the hypotheses of either Lemma~\ref{lem:stabilitySmooth} or~\ref{lem:stabilityNonsmooth}.  Suppose that $h_n\to0$ is any sequence such that $u^{h_n}$ converges uniformly to a continuous function $U$. Then $\langle U \rangle = 0$.
\end{lemma}
\begin{proof}
We recall first that $A^{h_n}(u^{h_n}) = 0$ by design~\eqref{eq:uh}.

Since $A^h$ is consistent on all Lipschitz functions and $u^{h_n}$ enjoy uniform Lipschitz bounds, we can also say that
\[ \abs{\langle u^{h_n} \rangle} = \abs{\langle u^{h_n} \rangle - A^{h_n}(u^{h_n})} \leq \tau(h_n). \]

Since convergence is uniform, the Dominated Convergence Theorem yields
\[ \langle U \rangle = \lim\limits_{n\to\infty}\langle u^{h_n} \rangle = 0. \qedhere  \]
\end{proof}

\begin{lemma}\label{lem:viscosity}
Let $u^h$ be defined by the schemes~\eqref{eq:vhScheme}-\eqref{eq:uh} and~\eqref{eq:uhSphere} under the hypotheses of either Lemma~\ref{lem:stabilitySmooth} or~\ref{lem:stabilityNonsmooth}.  Suppose that $h_n\to0$ is any sequence such that $u^{h_n}$ converges uniformly to a continuous function $U$. Then $U$ is a viscosity solution of~\eqref{eq:modifiedPDE}.
\end{lemma}
\begin{proof}
Here we follow the usual approach of the Barles-Souganidis framework, modified for the setting where the limit function is known to be continuous.  Recall that $u^h$ satisfies the scheme
\[ G^h(x,u(x)-u(\cdot)) + \tau(h)u^h(x) + \sigma(h) = 0 \]
where $\sigma(h)\to0$ as $h\to0$ (Lemma~\ref{lem:uhScheme}).  Moreover, there exists some $M\in\R$ such that $\|u^h\|_\infty \leq M$ for all sufficiently small $h>0$ (Lemmas~\ref{lem:stabilitySmooth}-\ref{lem:stabilityNonsmooth}).

Consider any $x_0\in \Sf$ and $\phi\in C^\infty$ such that $U-\phi$ has a strict local maximum at $x_0$ with $U(x_0) = \phi(x_0)$.  Because $u^h$ and the limit function $U$ are continuous, strict maxima are stable and there exists a sequence $z_n\in\G^h\cap \Sf$ such that
\[ z_n\to x_0,  \quad u^{h_n}(z_n)\to U(x_0) \]
where $z_n$ maximizes $u^{h_n}(x)-\phi(x)$ over points $x\in\G^h\cap\Sf$.


From the definition of $z_n$ as a maximizer of $u^{h_n}-\phi$, we also observe that
\[ u^{h_n}(z_n)-u^{h_n}(\cdot) \geq \phi(z_n)-\phi(\cdot). \]

We let $G(\nabla u(x),D^2u(x))$ denote the PDE operator~\eqref{eq:modifiedPDE}.  Since $u^{h_n}$ is a solution of the scheme, we can use monotonicity to calculate
\begin{align*} 0 &= G^{h_n}(z_n,u^{h_n}(z_n)-u^{h_n}(\cdot))+ \tau(h_n)u^{h_n}(z_n) + \sigma(h_n)\\
  &\geq G^{h_n}(z_n,\phi(z_n)-\phi(\cdot)) - M\tau(h_n) + \sigma(h_n). 
\end{align*}

As the scheme is consistent, we conclude that
\begin{align*} 0 &\geq \liminf\limits_{n\to\infty}\left(G^{h_n}(z_n,\phi(z_n)-\phi(\cdot)) - M\tau(h_n)+\sigma(h_n)\right)\\ &\geq G_*(x_0,\nabla\phi(x_0),D^2\phi(x_0)). \end{align*}
Thus $U$ is a subsolution of~\eqref{eq:modifiedPDE}. 

An identical argument shows that $U$ is a supersolution and therefore a viscosity solution.
\end{proof}

These lemmas lead immediately to our main convergence theorem.  The requirements on the schemes for the smooth and non-smooth setting are slightly different, but the proofs of the following two theorems are the same.
\begin{theorem}[Convergence (smooth case)]\label{thm:convergenceSmooth}
Consider the schemes~\eqref{eq:vhScheme}-\eqref{eq:uh} and~\eqref{eq:uhSphere} under the conditions of Hypothesis~\ref{hyp:Smooth}, \ref{hyp:grid}, and~\ref{hyp:schemes}.  Suppose also that~\eqref{eq:modifiedPDE} has a unique mean-zero $C^{0,1}$ solution.  Then $u^h$ converges uniformly to the unique smooth solution of~\eqref{eq:OTPDE}.
\end{theorem}
\begin{theorem}[Convergence (non-smooth case)]\label{thm:convergenceNonsmooth}
Consider the schemes~\eqref{eq:vhScheme}-\eqref{eq:uh} and~\eqref{eq:uhSphere} under the conditions of Hypothesis~\ref{hyp:Nonsmooth}, \ref{hyp:grid}, and~\ref{hyp:schemes}.  Suppose also that $F^h$ is an underestimating scheme and that~\eqref{eq:modifiedPDE} has a unique mean-zero $C^{0,1}$ solution.  Then $u^h$ converges uniformly to the unique Lipschitz continuous solution of~\eqref{eq:OTPDE}.
\end{theorem}
\begin{proof}
Consider any sequence $h_n\to0$.  Notice that the function $u^{h_n}$ is uniformly bounded (Lemmas~\ref{lem:stabilitySmooth}-\ref{lem:stabilityNonsmooth}) and enjoys uniform Lipschitz bounds (Lemma~\ref{lem:lipschitzSphere}).  Then by the Arzel\`{a}-Ascoli theorem there exists a subsequence $h_{n_k}$ and a continuous function $U$ such that $u^{h_{n_k}}$ converges uniformly to $U$, where $U$ has Lipschitz constant $L$.

From Lemmas~\ref{lem:meanZero}-\ref{lem:viscosity}, $U$ is a mean-zero viscosity solution of~\eqref{eq:modifiedPDE}.  Then by Theorems~\ref{thm:equivalenceSmooth} and~\ref{thm:equivalenceNonsmooth}, $U$ must agree with the unique mean-zero solution $u$ of~\eqref{eq:OTPDE}.

Since this holds for any sequence $h_n$, we conclude that $u^h$ converges uniformly to $u$.
\end{proof}

\section{Conclusion}\label{sec:conclusion}
We have constructed a convergence framework for numerically solving the optimal transport problem on the sphere.
This is done via a \MA-type PDE formulation.  This framework applies to both the squared geodesic cost, which has direct applications to moving-mesh methods on the sphere which have recently been used in meteorology problems, and to the logarithmic cost coming from the reflector antenna problem.

Our convergence framework is inspired by the Barles-Souganidis framework, but requires considerable consideration of the spherical geometry and the fact that there is no comparison principle for this PDE.  The convergent result applies to very general meshes and point clouds on the sphere, which need only satisfy a very mild regularity condition.  The convergence framework applies very generally to consistent, monotone approximation schemes.  By introducing appropriate local coordinates, the PDE can be locally posed on tangent planes, which allows for the use of a wide range of monotone approximation schemes for PDEs in $\R^2$. In addition, the advent of Lipschitz control in the PDE introduced sufficient stability to guarantee convergence.

The general convergence theorem guarantees convergence of the numerical method to the solution of the optimal transport problem when the data is sufficiently regular.  However, in the case of the squared geodesic cost we can further utilize the theory of viscosity solutions to guarantee the convergence of consistent, under-estimating schemes to non-smooth solutions.

In a companion paper~\cite{HT_SphereNumerics}, we show how to produce a particular finite difference implementation that fits with this convergence framework. Perhaps most importantly, this shows how to actually construct a practical, convergent scheme for the fully nonlinear \MA-type operator on the sphere. We also show how the convergence framework can be modified to accommodate the logarithmic cost by introducing a cutoff the makes this cost function Lipschitz.

\bibliographystyle{plain}
\bibliography{OTonSphere}

{\appendix
\section{Regularity}\label{app:regularity}

The results from Loeper~\cite{Loeper_OTonSphere} indicate that we have two r\'{e}gimes of regularity: classical and nonsmooth, both encapsulated in Theorem 2.4 of that paper. The classical result, adapted to our notation, is as follows:
\begin{theorem}[Regularity (smooth)]\label{thm:regSmooth}
Given data satisfying Hypothesis~\ref{hyp:Smooth}, suppose additionally that $f_1$ and $f_2$ are in $C^{1,1}(\Sf)$ (resp. $C^{\infty}(\Sf)$).  Then $u\in C^{3,\alpha}(\Sf)$ for every $\alpha \in [0,1)$ (resp. $u \in C^{\infty}(\Sf)$).
\end{theorem}

The corresponding non-smooth result is:
\begin{theorem}[Regularity (non-smooth)]\label{thm:regNonsmooth}
Given data satsifying Hypothesis~\ref{hyp:Nonsmooth}, suppose additionally that there exists some $h: \mathbb{R}^{+} \rightarrow \mathbb{R}^{+}$ with $\lim_{\epsilon \rightarrow 0} h(\epsilon) = 0$ such that
\begin{equation}
\int_{B_\epsilon(x)}f_1(y)\,dy \leq h(\epsilon) \epsilon, \quad \text{for all }  \epsilon \geq 0, x \in \Sf.
\end{equation}
Then $u \in C^1(\Sf)$.
\end{theorem}

As pointed out in Loeper~\cite{Loeper_OTonSphere}, this condition is automatically satisfied for densities $f_1\in L^p(\Sf)$ with $p>2$.  In fact, a slightly stronger regularity result is available in this case, and we have $u \in C^{1,\beta}(\Sf)$ with $\beta = \frac{p-2}{7p-2}$. The following Lemma will complete the proof of Theorem~\ref{thm:regularity} by showing that Theorem~\ref{thm:regNonsmooth} also applies to densities $f_1 \in L^1(\Sf)$.

\begin{lemma}[Integrability condition for $L^1$ densities]\label{lem:intL1}
If $\mu \in L^1(\Sf)$, then there exists some $h: \mathbb{R}^{+} \rightarrow \mathbb{R}^{+}$ with $\lim_{\epsilon \rightarrow 0} h(\epsilon) = 0$ such that
\[
\int_{B_\epsilon(x)}f_1(y)\,dy \leq h(\epsilon) \epsilon, \quad \text{for all }  \epsilon \geq 0, x \in \Sf.
\]
\end{lemma}
\begin{proof}
We use local spherical coordinates $\theta, \phi$ about the point $x$ to compute
\begin{align*}
\int_{B_\epsilon(x)}f_1(y)\,dy = \int_0^{2\pi}\int_0^\epsilon f(\theta,\phi)\sin\phi\, d\phi\, d\theta \leq \epsilon\int_0^{2\pi}\int_0^\epsilon f(\theta,\phi)\,d\phi\,d\theta,
\end{align*}
which holds for sufficiently small $\epsilon$ since then $\sin\phi<\phi\leq \epsilon$. By the Fubini-Tonelli Theorem, we can switch the order of integration and obtain
\[ \int_{B_\epsilon(x)}f_1(y)\,dy \leq \epsilon\int_0^\epsilon F(\phi)\,d\phi \]
where we have defined the partial integral
\[ F(\phi) = \int_0^{2\pi}f(\theta,\phi)\,d\theta. \]
Since $f$ is a non-negative $L^1$ function, the partial integral $F$ is also in $L^1$ and non-negative. We can then define
\[ h(\epsilon) = \int_0^\epsilon F(\phi)\,d\phi, \]
which satisfies $\lim_{\epsilon\to0} h(\epsilon)=0$ since $F\in L^1$.  Thus we obtain the desired result.
\end{proof}

\section{Mapping for the logarithmic cost}\label{app:logCost}
We calculate an explicit mapping $T(x,p)\in\Sf$ corresponding to the logarithmic cost $c(x,y) = -\log\norm{x-y}$. To accomplish this, we let $x\in\Sf$, $p\in\Tf(x)$ and solve~\eqref{eq:mapConditionSphere}:
\[\begin{cases}
\nabla_{\Sf,x}\log\norm{x-y} = p\\
\norm{y} = 1
\end{cases}\]
for $y$.

Let $\hat{\theta}$ and $\hat{\phi}$ be the local orthonormal tangent vectors at the point $x\in\Sf$.  Then we can compute this surface gradient in the ambient space in the local tangent coordinates using a simplified formula, which reduces the computational complexity:
\begin{equation}
\nabla_{\Sf} f(x) = \left( \nabla f(x) \cdot \hat{\theta}, \nabla f(x) \cdot \hat{\phi} \right)
\end{equation}
where we emphasize here that the gradient $\nabla$ refers to the usual gradient in $\R^3$. Using this formula, we obtain
\[
p = \left( \frac{(x-y) \cdot \hat{\theta}}{\norm{x-y}^2}, \frac{(x-y) \cdot \hat{\phi}}{\norm{x-y}^2} \right)
\]

Note that $x, \hat{\theta},$ and $\hat{\phi}$ form an orthonormal set. Thus we can express the unknown $y$ in the form
$ y = y_x x + y_\theta\hat{\theta} + y_\phi \hat{\phi} $
and obtain
\[
p = \left( \frac{-y_{\theta}}{2-2y_x}, \frac{-y_{\phi}}{2-2y_x} \right).
\]
Combining this with the requirement that $y_x^2+y_\theta^2+y_\phi^2 = \norm{y}^2 = 1$ allows us to solve for the components of $y$:
\begin{align*} y &= (y_x, y_\theta, y_\phi)\\  
  &= \frac{1}{4\norm{p}^2+1}\left(4\norm{p}^2-1,-4p_\theta,-4p_\phi\right)\\
	&= x\frac{\norm{p^2}-1/4}{\norm{p}^2+1/4}-\frac{p}{\norm{p}^2+1/4}.
\end{align*}

\section{Geodesic normal coordinates}\label{app:normalCoords}
Consider a point $x_0\in\Sf$ and the corresponding tangent plane $\Tf_{x_{0}}$.  Geodesic normal coordinates for points $x\in\Sf$ will take the form
\[ v_{x_0}(x) = x_0 + k\text{Proj}_{\Tf_{x_{0}}}(x-x_0) \in \Tf_{x_{0}} \]
where $k$ is chosen so that $\norm{x_0-v_{x_0}(x)} = d_{\Sf}(x_0,x)$.

Recall that the geodesic distance between $x$ and $x_0$ can be expressed as
\[ d_{\Sf}(x_0,x) = 2\arcsin\left(\frac{\norm{x-x_0}}{2}\right). \]
Since $x$ and $x_0$ are unit vectors, we can let $\cos\alpha = x\cdot x_0$ and compute
\[ \cos d_{\Sf}(x_0,x) = \cos\left(2\arcsin\left(\frac{\sqrt{2-2\cos\alpha}}{2}\right)\right) = \cos\alpha = x\cdot x_0.\]

We will make use of the unit tangent vectors $\hat{\theta}$ and $\hat{\phi}$ at the point $x_0$, which define orthonormal coordinates.  The projection of the displacement $x-x_0$ onto the tangent plane can be represented in these coordinates as
\[ \text{Proj}_{\Tf_{x_{0}}}(x-x_0) = \left[(x-x_0)\cdot\hat{\theta}\right]\hat{\theta} + \left[(x-x_0)\cdot\hat{\phi}\right]\hat{\phi}. \]

By computing a unit vector in this direction and scaling by the geodesic distance $d_{\Sf}(x_0,x)$, we obtain the following expression for the geodesic normal coordinates:
\[ v_{x_0}(x) = x_0 + d_{\Sf}(x_0,x)\frac{ \left[(x-x_0)\cdot\hat{\theta}\right]\hat{\theta} + \left[(x-x_0)\cdot\hat{\phi}\right]\hat{\phi}}{\sqrt{\left[(x-x_0)\cdot\hat{\theta}\right]^2 + \left[(x-x_0)\cdot\hat{\phi}\right]^2}}. \]

Since $x_0$ is a unit vector orthogonal to both $\hat{\theta}$ and $\hat{\phi}$, the actual displacement between points on the sphere can be expressed as
\[ x-x_0 = \left[(x-x_0)\cdot\hat{\theta}\right]\hat{\theta} + \left[(x-x_0)\cdot\hat{\phi}\right]\hat{\phi} + \left[(x-x_0)\cdot x_0\right] x_0, \]
which has squared Euclidean length
\[ \norm{x-x_0}^2 = \left[(x-x_0)\cdot\hat{\theta}\right]^2 + \left[(x-x_0)\cdot\hat{\phi}\right]^2 + \left[(x-x_0)\cdot x_0\right]^2. \]

These relationships allow us to simplify the expression for geodesic normal coordinates to
\begin{align*}
v_{x_0}(x) &= x_0 + d_{\Sf}(x_0,x)\frac{x-x_0 -\left[(x-x_0)\cdot x_0\right] x_0}{\norm{x-x_0}^2-\left[(x-x_0)\cdot x_0\right]^2}\\
  &= x_0 + d_{\Sf}(x_0,x) \frac{x-x_0(x\cdot x_0)}{\sqrt{1-(x\cdot x_0)^2}}\\
	&= x_0 + d_{\Sf}(x_0,x) \frac{x-x_0\cos d_{\Sf}(x_0,x)}{\sqrt{1-\cos^2 d_{\Sf}(x_0,x)}}\\
	&= x_0\left(1-d_{\Sf}(x_0,x)\cot d_{\Sf}(x_0,x)\right) + x \left(d_{\Sf}(x_0,x)\csc d_{\Sf}(x_0,x)\right).
\end{align*}

}

\end{document}